\documentclass[final,leqno,onefignum,onetabnum]{siamltex1213}
\usepackage{graphicx,epsfig,color}
\usepackage{booktabs}
\usepackage[notref,notcite]{showkeys}
\usepackage{amsmath,cases}
\usepackage{amssymb,bm,mathdots}
\usepackage{multirow}%,pxfonts}

\input{siam.sty}

\title{On well-conditioned spectral collocation and spectral methods by the integral reformulation%\thanks{Version date: 2015.07.01}
} 

\author{Kui Du\thanks{%Corresponding author. 
School of Mathematical Sciences and Fujian Provincial Key Laboratory of Mathematical Modelling and High-Performance Scientific Computation, Xiamen University, Xiamen 361005, China ({kuidu@xmu.edu.cn}). The research of this author was supported by the National Natural Science Foundation of China (No.11201392 and No.91430213), the Doctoral Fund of Ministry of Education of China (No.20120121120020) and the Fundamental Research Funds for the Central Universities (No.2013121003).} 
}

\begin{document}
\maketitle
\slugger{mms}{xxxx}{xx}{x}{x--x}%slugger should be set to mms, siap, sicomp, sicon, sidma, sima, simax, sinum, siopt, sisc, or sirev

\begin{abstract} Well-conditioned spectral collocation and spectral methods have recently been proposed to solve differential equations. In this paper, we revisit the well-conditioned spectral collocation methods proposed in [T.~A. Driscoll, {\it J. Comput. Phys.}, 229 (2010), pp.~5980-5998] and [L.-L. Wang, M.~D. Samson, and X.~Zhao, {\it SIAM J. Sci. Comput.}, 36 (2014), pp.~A907--A929], and the ultraspherical spectral method proposed in [S.~Olver and A.~Townsend, {\it SIAM Rev.}, 55 (2013), pp.~462--489] for an $m$th-order ordinary differential equation from the viewpoint of the integral reformulation. Moreover, we propose a Chebyshev spectral method for the integral reformulation. The well-conditioning of these methods is obvious by noting that the resulting linear operator is a compact perturbation of the identity. The adaptive QR approach for the ultraspherical spectral method still applies to the almost-banded infinite-dimensional system arising in the Chebyshev spectral method for the integral reformulation. Numerical examples are given to confirm the well-conditioning of the Chebyshev spectral method.   
\end{abstract}



\pagestyle{myheadings}
\thispagestyle{plain}
\markboth{KUI DU}{On well-conditioned spectral collocation and spectral methods}

\section{Introduction} In recent decades, spectral collocation and spectral methods have been extensively used for solving differential equations because of their high order of accuracy; see, for example, %gottlieb1984spect,
\cite{fornberg1996pract,trefethen2000spect,boyd2001cheby,canuto2006spect,shen2011spect} and the references therein. However, classical spectral collocation and spectral methods for differential equations lead to ill-conditioned matrices. For example, the condition number of the matrix in the rectangular spectral collocation method \cite{driscoll2015recta} for an $m$th-order differential operator grows like $\mcalo(N^{2m})$, where $N$ is the number of collocation points. Preconditioning techniques (see, for example, %\cite{coutsias1996integ,coutsias1996effic,hesthaven1998integ,elbarbary2006integ} 
\cite{coutsias1996integ,coutsias1996effic,kim1996preco,kim1997preco,hesthaven1998integ,parter2001preco,parter2001precon,elbarbary2006integ}) and spectral integration techniques (see, for example, \cite{greengard1991spect,lee1997fast,driscoll2010autom,elgindy2013solvi,viswanath2015spect}) have been employed to improve the conditioning.

In this paper, from the viewpoint of the integral reformulation, we revisit recently proposed well-conditioned spectral collocation methods \cite{driscoll2010autom,wang2014well} and the ultraspherical spectral method \cite{olver2013fast} for solving the $m$th-order linear ordinary differential equation \beq\label{code}u^{(m)}(x)+\sum_{k=0}^{m-1}a^{k}(x)u^{(k)}(x)=f(x),\eeq together with $m$ linearly independent constraints \beq\label{mlc}\mcalb u={\bf b}.\eeq Here, $a^k(x)$, $u(x)$ and $f(x)$ are suitably smooth functions defined on $[-1,1]$ and ${\bf b}$ denotes a constant $m$-vector. 

Define the integral operators: $$\p_x^{-1} \phi(x)=\int_{-1}^x \phi(t)\rmd t; \qquad \p_x^{-k}\phi(x)=\p_x^{-1}\l(\p_x^{-(k-1)}\phi(x)\r), \quad k\geq 2.$$ Let $v(x)=u^{(m)}(x)$. Employing (\ref{mlc}), we can write 
\beqs\label{intv}u(x)=\p_x^{-m}v(x)+\mcalx(\mcalb\mcalx)^{-1}({\bf b}-\mcalb\p_x^{-m}v(x)),\eeqs  where $$\mcalx=\l[\begin{array}{cccc}1&x^1&\cdots&x^{m-1}\end{array}\r].$$ Here, $\mcalx$ can be replaced by any basis for the vector space of polynomials of degree less than $m$. We use this $\mcalx$ for clarity. The problem (\ref{code})-(\ref{mlc}) can be written as the following integral reformulation \beq\label{intref2}v(x)+\sum_{k=0}^{m-1}a^k(x)\p_x^{k-m}v(x)-\mcala(\mcalb\mcalx)^{-1}\mcalb\p_x^{-m}v(x)=f(x)-\mcala(\mcalb\mcalx)^{-1}{\bf b},\eeq where $$\mcala=\l[\begin{array}{cccc} a^0(x)& a^0(x)x+a^1(x)& \cdots & \sum_{k=0}^{m-1}\frac{(m-1)!}{(m-1-k)!}a^k(x)x^{m-1-k}\end{array}\r].$$ We rewrite (\ref{intref2}) as a linear operator equation \beq\label{oe}(\mcali-\mcalk)v(x)=g(x),\eeq where $\mcali$ denotes the identity operator, $$\mcalk v=-\sum_{k=0}^{m-1}a^k(x)\p_x^{k-m}v(x)+\mcala(\mcalb\mcalx)^{-1}\mcalb\p_x^{-m}v(x),$$ and $$g(x)=f(x)-\mcala(\mcalb\mcalx)^{-1}{\bf b}.$$ If the linear operator $\mcalb$ is bounded, then it is easy to show that $\mcalk$ is a compact operator. As noted in \cite{rokhlin1983solut,greengard1991spect,greengard1991numer,starr1994numer}, the linear system in the Galerikin method \cite{ikebe1972galer} for (\ref{oe}) is well-conditioned if both $\|\mcali-\mcalk\|$ and $\|(\mcali-\mcalk)^{-1}\|$ are  $\mcalo(1)$. 
\cite{greengard1991numer,starr1994numer}
The automatic spectral collocation method \cite{driscoll2010autom} for the problem (\ref{code})-(\ref{mlc}) is a Chebyshev spectral collocation scheme based on the integral equation (\ref{intref2}), which has been implemented in the Chebfun software system \cite{driscoll2014chebf}. In this paper, we propose a Chebyshev spectral method for the integral reformulation (\ref{intref2}).

The main contributions of this paper are in order. (I): Solving the Birkhoff interpolation problem (\ref{birkhoff}) by the integral formulation (\ref{form}) and deriving it from a right preconditioning rectangular spectral collocation method, we show that the well-conditioned spectral collocation method \cite{wang2014well} is essentially a stable implementation of the spectral collocation discretization of the integral equation (\ref{intref2}). (II): We show that the Chebyshev spectral method proposed in this work for the integral equation (\ref{intref2}) leads to an almost-banded infinite-dimensional system (see (\ref{scop1})), which can be solved by the adaptive QR approach in \cite{olver2013fast,olver2014pract}. Because only Chebyshev series are involved in the Chebyshev spectral method, no conversion operators $\mcals_k$ (see \S \ref{usm}) are required. We only need the multiplication operator $\mcalm_0[a]$ (see \S \ref{usm}), which represents multiplication of two Chebyshev series. Therefore, the Chebyshev spectral method is very easy to implement. We show that the Chebyshev spectral method for the integral equation (\ref{intref2}) is a two-sided preconditioned version of the discrete operator equation (\ref{dcode}) in the ultraspherical spectral method \cite{olver2013fast} for the differential problem  (\ref{code})-(\ref{mlc}). The Chebyshev spectral method is obviously well-conditioned   and avoids the drawbacks mentioned in \cite[\S 1.1.4]{olver2013fast}. Finally, for large bandwidth variable coefficients, iterative solvers can be employed to achieve computational efficiency because the involved matrix-vector product in the Chebyshev spectral method can be obtained in $\mcalo(n\log n)$ operations, where $n$ is the truncation parameter.
 
The rest of the paper is organized as follows. In \S 2, we derive the well-conditioned spectral collocation method \cite{wang2014well} for (\ref{code})-(\ref{mlc}) by utilizing a right preconditioning approach and show that it is a stable implementation of the spectral collocation discretization of the integral equation (\ref{intref2}). In \S 3, we review the ultraspherical spectral method \cite{olver2013fast}. We also point out the relation (see (\ref{mulcon})) between the multiplication operator $\mcalm_0[a]$ representing multiplication of two Chebyshev series and the multiplication operator $\mcalm_k[a]$ representing multiplication of two ultraspherical ($C^{(k)}$) series. In \S 4, utilizing the multiplication operator $\mcalm_0[a]$ and the spectral integration operator \cite{greengard1991spect}, we present the Chebyshev spectral method  and show the relation with the ultraspherical spectral method. Numerical examples confirming the well-conditioning of the Chebyshev spectral method are given in \S 5. We present brief concluding remarks in \S 6.

\section{A well-conditioned spectral collocation method} 
%\subsection{Barycentric resampling matrix} 
%Throughout this section, 
Let $\{x_j\}_{j=0}^N$ be a set of points satisfying \beq\label{iptsx} -1\leq x_0<x_1<\cdots <x_{N-1}<x_N\leq 1,\eeq and $\{y_j\}_{j=0}^M$ be another set of distinct points satisfying \beq\label{iptsy}-1\leq y_0< y_1< \cdots <y_{M-1}<y_M\leq 1.\eeq 
The barycentric weights associated with the points $\{x_j\}_{j=0}^N$ are defined by \beq\label{bweight}w_{j,{\bf x}}=\prod_{n=0,n\neq j}^N(x_j-x_n)^{-1}, \qquad j=0,1,\ldots,N.\eeq 
The barycentric resampling matrix \cite{driscoll2015recta}, ${\bf P}^{{\bf x}\mapsto {\bf y}}\in{\mbbr^{(M+1)\times(N+1)}}$, which interpolates between the points $\{x_j\}_{j=0}^N$ and $\{y_j\}_{j=0}^M$, is defined by \beqs{\bf P}^{{\bf x}\mapsto {\bf y}}=[p^{{\bf x}\mapsto {\bf y}}_{ij}]_{i=0,j=0}^{M,N},\eeqs where \beqs p^{{\bf x}\mapsto {\bf y}}_{ij}=\l\{\begin{array}{ll} \dsp\frac{w_{j,{\bf x}}}{y_i-x_j}\l(\sum_{l=0}^N\dsp\frac{w_{l,{\bf x}}}{y_i-x_l}\r)^{-1}, & y_i\neq x_j, \\ 1, & y_i=x_j.
\end{array} \r.\eeqs %The matrix ${\bf P}^{N,m}$ is called a {\it downsampling}  ({\it upsampling}) matrix if $m>0$ $(m<0)$.
\begin{lemma}\label{mnnm}If $N\geq M$, then ${\bf P}^{{\bf x}\mapsto {\bf y}}{\bf P}^{{\bf y}\mapsto {\bf x}}={\bf I}_{M+1}.$ \end{lemma}

\subsection{Pseudospectral differentiation matrices}\label{sdm} 

The Lagrange interpolation basis polynomials of degree $N$ associated with the points $\{x_j\}_{j=0}^N$ are defined by \beqs\label{cf}  \ell_{j,{\bf x}}(x)=w_{j,{\bf x}}\prod_{n=0,n\neq j}^N(x-x_n),\qquad j=0,1,\ldots,N,\eeqs where $w_{j,{\bf x}}$ is the barycentric weight (\ref{bweight}). Define the pseudospectral differentiation matrices: \beqs\label{psdm} {\bf D}^{(k)}_{{\bf x}\mapsto {\bf x}}=\l[ k\ell_{j,{\bf x}}^{(k)}(x_i)\r]_{i,j=0}^N,\qquad {\bf D}^{(k)}_{{\bf x}\mapsto {\bf y}}=\l[ k\ell_{j,{\bf x}}^{(k)}(y_i)\r]_{i=0,j=0}^{M,N}.
\eeqs There hold $${\bf D}^{(k)}_{{\bf x}\mapsto {\bf y}}={\bf P}^{{\bf x}\mapsto {\bf y}}{\bf D}^{(k)}_{{\bf x}\mapsto {\bf x}},\qquad {\bf D}^{(0)}_{{\bf x}\mapsto {\bf y}}={\bf P}^{{\bf x}\mapsto {\bf y}},$$ and $${\bf D}^{(k)}_{{\bf x}\mapsto {\bf x}}=\l({\bf D}^{(1)}_{{\bf x}\mapsto {\bf x}}\r)^k,\qquad k\geq 1.$$

The matrix ${\bf D}^{(k)}_{{\bf x}\mapsto {\bf y}}$ is called a rectangular $k$th-order differentiation matrix, which maps values of a polynomial defined on $\{x_j\}_{j=0}^N$ to the values of its $k$th-order derivative on $\{y_j\}_{j=0}^M$. Explicit formulae and recurrences for pseudospectral differentiation matrices can be found in, for example, \cite{weideman2000matla,xu2015expli}. 

\subsection{Pseudospectral integration matrices} 
Given $\{y_j\}_{j=0}^{M}$ and ${\bf b}$ %$\{b_j\}_{j=0}^m$  
%${\bf b}=[\begin{array}{cccc}b_1 & b_2 & \cdots & b_m\end{array}]$ 
with $m\geq 1$, we consider the Birkhoff-type  interpolation problem: \beq\label{birkhoff} {\rm Find}\ p(x)\in\mbbp_{M+m}\ {\rm such \ that}\ \l\{ \begin{array}{ll} p^{(m)}(y_j)=u^{(m)}(y_j),& j=0,\ldots,M,  \\ \mcalb p={\bf b}.&\end{array}\r.\eeq Let $\{\ell_{j,{\bf y}}(x)\}_{j=0}^{M}$ be the Lagrange interpolation basis polynomials of degree $M$ associated with the points $\{y_j\}_{j=0}^{M}$. Then the Birkhoff-type interpolation polynomial takes the form \beq\label{form} p(x)=\sum_{j=0}^Mu^{(m)}(y_j) \p_x^{-m}\ell_{j,{\bf y}}(x)+\sum_{i=0}^{m-1}{\alpha_i}x^i,\eeq where $\alpha_i$ can be determined by the linear constraints $\mcalb p=\l[\begin{array}{cccc}b_1 & b_2 & \cdots & b_m\end{array}\r]^\rmt.$ Obviously, the existence and uniqueness of the Birkhoff-type interpolation polynomial is equivalent to that of $\{\alpha_i\}_{i=0}^{m-1}$. After obtaining $\alpha_i$, we can rewrite (\ref{form}) as \beqs\label{bip}p(x)=\sum_{j=0}^{M}u^{(m)}(y_j)B_{j,{\bf y}}(x)+\sum_{j=1}^{m}b_jB_{M+j,{\bf y}}(x).\eeqs %Obviously, each $B_{j,{\bf y}}(x)\in\mbbp_{M+m}$. 

Now we give some examples for $\{B_{j,{\bf y}}\}_{j=0}^{M+m}$. The first-order Birkhoff-type interpolation problem takes the form: \beqs\label{birkhoff1}{\rm Find} \ \ p(x)\in\mbbp_{M+1} \ \ {\rm such \ that} \ \ \l\{ \begin{array}{ll} p'(y_j)=u'(y_j),& j=0,1,\ldots,M,  \\ \mcalb p= b_1. & \end{array}\r.\eeqs  
\begin{itemize}
%\item Given $\mcall p:= ap(-1)+bp'(-1)$ with $a\neq 0$, we have \beqas && B_{j,{\bf y}}(x)=\p_x^{-1}\ell_{j,{\bf y}}(x)-\frac{b}{a}\ell_{j,{\bf y}}(-1), \quad j=0,1,\ldots,M,\\ &&  B_{M+1}(x)=\frac{1}{a}.\eeqas 

%\item Given $\mcall p:= ap(1)+bp'(1)$ with $a\neq 0$, we have \beqas && B_{j,{\bf y}}(x)=\p_x^{-1}\ell_{j,{\bf y}}(x)-\int_{-1}^1\ell_{j,{\bf y}}(x)\rmd x-\frac{b}{a}\ell_{j,{\bf y}}(1), \quad j=0,1,\ldots,M,\\ &&  B_{M+1}(x)=\frac{1}{a}.\eeqas 

\item Given $\mcalb p:= ap(-1)+bp(1)$  with $a+b\neq 0$, we have
\beqas && B_{j,{\bf y}}(x)=\p_x^{-1}\ell_{j,{\bf y}}(x)-\frac{b}{a+b}\int_{-1}^1\ell_{j,{\bf y}}(x)\rmd x, \quad j=0,1,\ldots,M,\\ &&  B_{M+1,{\bf y}}(x)=\frac{1}{a+b}.\eeqas 

\item Given $\dsp\mcalb p:= \int_{-1}^1 p(x)\rmd x$, we have \beqas && B_{j,{\bf y}}(x)=\p_x^{-1}\ell_{j,{\bf y}}(x)-\frac{1}{2}\int_{-1}^1\p_x^{-1}\ell_{j,{\bf y}}(x)\rmd x, \quad j=0,1,\ldots,M,\\ &&  B_{M+1,{\bf y}}(x)=\frac{1}{2}.\eeqas 
\end{itemize}
%By (\ref{ltp}) and (\ref{px1}), the matrix ${\bf B}^{(-1)}_{{\bf y}\mapsto {\bf x}}$ can be computed stably even for thousands of collocation points. 
%\subsection{Second-order case ($m=2$)}
The second-order Birkhoff-type interpolation problem takes the form: \beqs\label{birkhoff2}{\rm Find} \ \ p(x)\in\mbbp_{M+2} \ \ {\rm such \ that} \ \ \l\{ \begin{array}{ll} p''(y_j)=u''(y_j),& j=0,1,\ldots,M,  \\ \mcalb p={\bf b}. &\end{array}\r.\eeqs    
\begin{itemize}
%\item Given $\mcall_1(p,p'):= a_-p(-1)+b_-p'(-1),\qquad \mcall_2(p,p'):= a_+p(1)+b_+p'(1),$ with $$d:=2a_+a_--a_+b_-+a_-b_+\neq 0,$$ we have  \beqas && B_{j,{\bf y}}(x)=\l(\frac{b_-}{d}-\frac{a_-}{d}(1+x)\r)\int_{-1}^1\l(a_+\p_x^{-1}\ell_{j,{\bf y}}(x)+b_+\ell_{j,{\bf y}}(x)\r)\rmd x\\ && \qquad\qquad +\p_x^{-2}\ell_{j,{\bf y}}(x),\quad  j=0,1,\ldots,M,  \\ && B_{M+1}(x)=\frac{a_+}{d}(1-x)+\frac{b_+}{d}, \\ && B_{M+2}(x)=\frac{a_-}{d}(1+x)-\frac{b_-}{d}, \eeqas and \beqas && B_{j,{\bf y}}'(x)=-\frac{a_-}{d}\int_{-1}^1\l(a_+\p_x^{-1}\ell_{j,{\bf y}}(x)+b_+\ell_{j,{\bf y}}(x)\r)\rmd x\\ && \qquad\qquad +\p_x^{-1}\ell_{j,{\bf y}}(x),\quad  j=0,1,\ldots,M,  \\ && B_{M+1}'(x)=-\frac{a_+}{d}, \\ && B_{M+2}'(x)=\frac{a_-}{d}. \eeqas

\item Given \beq\label{lc2}\mcalb p=\l[\begin{array}{c}ap(-1)+bp(1)\\ \dsp\int_{-1}^1p(x)\rmd x  \end{array}\r],\qquad a\neq b,\eeq
%$\mcall_1(p,p'):=ap(-1)+bp(1)$ with $a\neq b$, and $\mcall_2(p,p')=\dsp\int_{-1}^1p(x)\rmd x$, 
we have 
\beqas && B_{j,{\bf y}}(x)=\p_x^{-2}\ell_{j,{\bf y}}(x)-\frac{bx}{b-a}\int_{-1}^1\p_x^{-1}\ell_{j,{\bf y}}(x)\rmd x
\\ &&\qquad\qquad +\l(\frac{(a+b)x}{2(b-a)}-\frac{1}{2}\r)\int_{-1}^1\p_x^{-2}\ell_{j,{\bf y}}(x)\rmd x,\quad j=0,1,\ldots,M, \\ && B_{M+1,{\bf y}}(x)=\frac{x}{b-a}, \\ && B_{M+2,{\bf y}}(x)=\frac{1}{2}-\frac{(a+b)x}{2(b-a)},
\eeqas and \beqas && B_{j,{\bf y}}'(x)=\p_x^{-1}\ell_{j,{\bf y}}(x)-\frac{b}{b-a}\int_{-1}^1\p_x^{-1}\ell_{j,{\bf y}}(x)\rmd x
\\ &&\qquad\qquad +\frac{a+b}{2(b-a)}\int_{-1}^1\p_x^{-2}\ell_{j,{\bf y}}(x)\rmd x,\quad j=0,1,\ldots,M, \\ && B_{M+1,{\bf y}}'(x)=\frac{1}{b-a}, \\ && B_{M+2,{\bf y}}'(x)=-\frac{a+b}{2(b-a)}.
\eeqas
\end{itemize} 
%By (\ref{ltp}), (\ref{px1}) and (\ref{px2}), the matrices ${\bf B}^{(-2)}_{{\bf y}\mapsto {\bf x}}$ and ${\bf B}^{(1-2)}_{{\bf y}\mapsto {\bf x}}$ can be computed stably even for thousands of collocation points. 

Let $N=M+m$ and $\{x_i\}_{i=0}^N$ be the points as in (\ref{iptsx}). Define the $m$th-order {\it pseudospectral integration matrix} (PSIM) as:
\beqs\label{psim}{\bf B}^{(-m)}_{{\bf y}\mapsto {\bf x}}=\l[B_{j,{\bf y}}(x_i)\r]_{i,j=0}^N.\eeqs Define the matrices $${\bf B}^{(k-m)}_{{\bf y}\mapsto {\bf x}}=\l[B_{j}^{(k)}(x_i)\r]_{i,j=0}^N,\qquad k\geq 1.$$ It is easy to show that \beqs\label{dkqm}{\bf B}^{(k-m)}_{{\bf y}\mapsto {\bf x}}={\bf D}^{(k)}_{{\bf x}\mapsto {\bf x}}{\bf B}^{(-m)}_{{\bf y}\mapsto {\bf x}},\qquad k\geq 1.\eeqs
The matrices ${\bf B}^{(-m)}_{{\bf y}\mapsto {\bf x}}$ and ${\bf B}^{(k-m)}_{{\bf y}\mapsto {\bf x}}$ can be computed stably even for thousands of collocation points; see \cite{wang2014well,du2015prersc} for details.

%Let ${\bf L}_\mcalb$ be the discretization of the linear operator $\mcalb$. We have the following theorem.
 
\begin{theorem}\label{main} Let ${\bf L}_\mcalb$ be the discretization of the linear operator $\mcalb$. If for any $p(x)\in\mbbp_N$, \beq\label{discond}\mcalb p= {\bf L}_\mcalb\l[\begin{array}{c} p(x_0) \\ \vdots \\ p(x_N)\end{array}\r],\eeq then \beqs\label{identity} \l[\begin{array}{c}{\bf D}^{(m)}_{{\bf x}\mapsto {\bf y}}\\ {\bf L}_\mcalb\end{array}\r]{\bf B}^{(-m)}_{{\bf y}\mapsto {\bf x}}={\bf I}_{N+1}. \eeqs \end{theorem} \begin{proof} The result follows from \beqs{\bf D}^{(m)}_{{\bf x}\mapsto {\bf x}}{\bf B}^{(-m)}_{{\bf y}\mapsto {\bf x}}=\l[\begin{array}{ll} {\bf P}^{{\bf y}\mapsto {\bf x}} & {\bm 0}\end{array}\r],\qquad {\bf L}_\mcalb{\bf B}^{(-m)}_{{\bf y}\mapsto {\bf x}}=\l[\begin{array}{ll}{\bm 0}& {\bf I}_m\end{array}\r],\eeqs and Lemma \ref{mnnm}. \end{proof}
 
Now we give a concrete example for ${\bf L}_\mcalb$ satisfying (\ref{discond}). Let $\mcalb$ be given by (\ref{lc2}). %Consider the linear constraint \beq\label{lc1}ap(-1)+bp(1)=b_1,\eeq and the linear constraint \beq\label{lc2} \int_{-1}^1p(x)\rmd x=b_2,\eeq where $a$, $b$, $b_1$ and $b_2$ are given constants. 
It is straightforward to discretize: $${\bf L}_\mcalb=\l[\begin{array}{ccccc}a&0 &\cdots & 0 & b\\ \omega_0 & \omega_1 & \cdots &\cdots & \omega_N \end{array}\r],$$ where $\{\omega_j\}_{j=0}^N$ are Clenshaw-Curtis quadrature weights \cite{funaro1992polyn}.

\subsection{The method}
Let $\{x_j\}_{j=0}^N$ (with $N=M+m$) and $\{y_j\}_{j=0}^M$ be the points as  defined in (\ref{iptsx}) and (\ref{iptsy}), respectively.
The rectangular spectral collocation method \cite{driscoll2015recta} for (\ref{code})-(\ref{mlc}) finds a column vector ${\bf u}$ as an approximation of the vector of the exact solution $u(x)$ at the points $\{x_j\}_{j=0}^N$, i.e., $${\bf u}\approx\l[\begin{array}{cccc} u(x_0) & u(x_1) & \cdots & u(x_N)\end{array}\r]^\rmt.$$ The collocation scheme for (\ref{code}) is given by $${\bf A}_{M+1}{\bf u}={\bf f},$$ where $${\bf A}_{M+1}={\bf D}^{(m)}_{{\bf x}\mapsto {\bf y}}+\diag\{{\bf a}^{m-1}\}{\bf D}^{(m-1)}_{{\bf x}\mapsto {\bf y}}+\cdots+
\diag\{{\bf a}^1\}{\bf D}^{(1)}_{{\bf x}\mapsto {\bf y}}+\diag\{{\bf a}^0\}{\bf D}^{(0)}_{{\bf x}\mapsto {\bf y}}.$$ Here we use boldface letters to indicate a column vector obtained by discretizing at the points $\{y_j\}_{j=0}^M$ except for the unknown ${\bf u}$. For example, \beqas &&{\bf a}^0=\l[\begin{array}{cccc}a^0(y_0)& a^0(y_1) & \cdots & a^0(y_M)\end{array}\r]^\rmt, \\ &&{\bf f}=\l[\begin{array}{cccc}f(y_0)& f(y_1) & \cdots & f(y_M)\end{array}\r]^\rmt.\eeqas Let ${\bf L}_\mcalb{\bf u}={\bf b}$ be the discretization of the linear constraints (\ref{mlc}) and satisfy the condition (\ref{discond}). The global collocation system is given by \beq\label{gcs} {\bf Au=g}\eeq where \beqas &&{\bf A}=\l[\begin{array}{c}{\bf A}_{M+1}\\ {\bf L}_\mcalb\end{array}\r],\quad  {\bf g}=\l[\begin{array}{c} {\bf f}\\ {\bf b}\end{array}\r].\eeqas

Consider the pseudospectral integration matrix ${\bf B}^{(-m)}_{{\bf y}\mapsto {\bf x}}$ (\ref{psim}) as a right preconditioner for the linear system (\ref{gcs}). We need to solve the right preconditioned linear system \beqs{\bf A}{\bf B}^{(-m)}_{{\bf y}\mapsto {\bf x}}{\bf v}={\bf g}.\eeqs 
There hold $${\bf A}_{M+1}{\bf B}^{(-m)}_{{\bf y}\mapsto {\bf x}}\l[\begin{array}{c}{\bf I}_{M+1} \\ {\bm 0}\end{array}\r]={\bf I}_{M+1}+\diag\{{\bf a}_{m-1}\}\wt{\bf B}^{(m-1-m)}_{{\bf y}\mapsto {\bf y}}+\cdots+\diag\{{\bf a}_0\}\wt{\bf B}^{(0-m)}_{{\bf y}\mapsto {\bf y}},$$ %\beqas&& {\bf A}_{M+1}{\bf B}^{(-m)}_{{\bf y}\mapsto {\bf x}}\l[\begin{array}{c}{\bf I}_{M+1} \\ {\bm 0}\end{array}\r]\\&=&{\bf I}_{M+1}+\diag\{{\bf a}_{m-1}\}\wt{\bf B}^{(m-1-m)}_{{\bf y}\mapsto {\bf y}}+\cdots+\diag\{{\bf a}_0\}\wt{\bf B}^{(0-m)}_{{\bf y}\mapsto {\bf y}},\eeqas 
and \beqs
{\bf A}_{M+1}{\bf B}^{(-m)}_{{\bf y}\mapsto {\bf x}}\l[\begin{array}{c}{\bm 0}\\ {\bf I}_{m} \end{array}\r]= \diag\{{\bf a}_{m-1}\}\wh{\bf B}^{(m-1-m)}_{{\bf y}\mapsto {\bf y}}+\cdots+\diag\{{\bf a}_0\}\wh{\bf B}^{(0-m)}_{{\bf y}\mapsto {\bf y}},\eeqs where, for $k=0,1,\cdots,m-1$, $$\wt{\bf B}^{(k-m)}_{{\bf y}\mapsto {\bf y}}=\l[B_{j,{\bf y}}^{(k)}(y_i)\r]_{i,j=0}^M,\qquad \wh{\bf B}^{(k-m)}_{{\bf y}\mapsto {\bf y}}=\l[B_{j,{\bf y}}^{(k)}(y_i)\r]_{i=0,j=M+1}^{M,M+m}.$$ By (see Theorem \ref{main}) $${\bf L}_\mcalb{\bf B}^{(-m)}_{{\bf y}\mapsto {\bf x}}=\l[\begin{array}{ll}{\bm 0}& {\bf I}_m\end{array}\r],$$ we have %\beq\label{modal1}{\bf v}=\l[\begin{array}{c} {\bf v}_{M+1} \\ {\bf c}_m \end{array}\r]\eeq and 
$${\bf v}=\l[\begin{array}{c} {\bf v}_{M+1} \\ {\bf b} \end{array}\r],$$ and
\beq\label{modal2}\wt{\bf A}_{M+1}{\bf v}_{M+1}=\wt{\bf f},\eeq
where $$\wt{\bf A}_{M+1}={\bf I}_{M+1}+\diag\{{\bf a}_{m-1}\}\wt{\bf B}^{(m-1-m)}_{{\bf y}\mapsto {\bf y}}+\cdots+\diag\{{\bf a}_0\}\wt{\bf B}^{(0-m)}_{{\bf y}\mapsto {\bf y}},$$ and $$\wt{\bf f}={\bf f}-\l(\diag\{{\bf a}_{m-1}\}\wh{\bf B}^{(m-1-m)}_{{\bf y}\mapsto {\bf y}}+\cdots+\diag\{{\bf a}_0\}\wh{\bf B}^{(0-m)}_{{\bf y}\mapsto {\bf y}}\r){\bf b}.$$
Obviously, ${\bf v}_{M+1}$ is an approximation of the vector $${\bf v}_{M+1}\approx\l[\begin{array}{cccc}u^{(m)}(y_0) & u^{(m)}(y_1) & \cdots & u^{(m)}(y_M)\end{array}\r]^\rmt.$$ After solving (\ref{modal2}), we obtain ${\bf u}$ by $${\bf u}={\bf B}^{(-m)}_{{\bf y}\mapsto {\bf x}}\l[\begin{array}{c} {\bf v}_{M+1} \\ {\bf b} \end{array}\r].$$

By the above derivation, the linear system $(\ref{modal2})$ is the same as that obtained by the following collocation scheme for the integral reformulation $(\ref{intref2})$: (i) substitute $v(x)=\sum_{j=0}^Mu^{(m)}(y_j)\ell_{j,{\bf y}}(x)$ into $(\ref{intref2})$; (ii) collocate the resulting integral equation on points $\{y_j\}_{j=0}^M$. It follows from this observation that the condition number of the coefficient matrix $\wt {\bf A}_{M+1}$ in $(\ref{modal2})$ is independent of the number of collocation points. Numerical experiments showing this point are reported in \cite{wang2014well,du2015prersc}.

\section{The ultraspherical spectral method}\label{usm} 
The ultraspherical spectral method \cite{olver2013fast} finds an infinite vector $${\bf u}=\l[\begin{array}{ccc}u_0 & u_1 & \cdots\end{array}\r]^\rmt$$ such that the Chebyshev expansion of the solution of (\ref{code})-(\ref{mlc}) is given by $$u(x)=\sum_{j=0}^\infty u_j T_j(x),\qquad x\in[-1,1],$$ where $T_j(x)$ is the degree $j$ Chebyshev polynomial \cite{funaro1992polyn,olver2010nist}. We review the details of this method in the following.

Note that the following recurrence relation $$\frac{\rmd^ k T_n}{\rmd x^ k}=\l\{\begin{array}{ll} 2^{ k-1}n( k-1)!C_{n- k}^{( k)}, & n\geq k,\\ 0, & 0\leq n\leq  k-1, \end{array}\r.$$ where $C_j^{( k)}$ is the ultraspherical polynomial with an integer parameter $ k\geq 1$ of degree $j$ \cite{funaro1992polyn,olver2010nist}.
Then the differentiation operator for the $ k$th derivative is given by \beq\label{l0}\mcald_ k=2^{ k-1}( k-1)!\l[\begin{array}{ccccc} {\bf 0}& k&&& \\ && k+1&&\\ &&& k+2&\\ &&&&\ddots\end{array}\r], \quad  k\geq 1.\eeq Here $\bf 0$ in (\ref{l0})
denotes a $ k$-dimensional zero row vector. 

The conversion operator converting a vector of Chebyshev expansion coefficients to a vector of $C^{(1)}$ expansion coefficients, denoted by $\mcals_0$, and the conversion operator converting a vector of $C^{( k)}$ expansion coefficients to a vector of $C^{( k+1)}$ expansion coefficients, denoted by $\mcals_ k$, are given by $$\mcals_0=\l[\begin{array}{ccccc}1 &0 &-\frac{1}{2}&& \\ &\frac{1}{2}&0&-\frac{1}{2}&\\ &&\frac{1}{2}&0 &\ddots\\ &&&\frac{1}{2}&\ddots\\ &&&&\ddots \end{array}\r],\quad \mcals_ k=\l[\begin{array}{ccccc}1 &0 &-\frac{ k}{ k+2}&& \\ &\frac{ k}{ k+1}&0&-\frac{ k}{ k+3}&\\ &&\frac{ k}{ k+2}&0 &\ddots\\ &&&\frac{ k}{ k+3}&\ddots\\ &&&&\ddots \end{array}\r],\  k\geq 1.$$ 
%$$\|\mcals_0\|_2\approx1.1547,\qquad \|\mcals_1\|_2\approx1.0601,\qquad \|\mcals_2\|_2\approx1.1473$$ $$\|\mcals_3\|_2\approx1.2246,\qquad \|\mcals_4\|_2\approx1.2885,\qquad \|\mcals_5\|_2\approx1.3410$$ $$\|\mcals_\infty\|=2.$$

We also require the multiplication operator $\mcalm_0[a]$ that represents multiplication of two Chebyshev series, and the multiplication operator $\mcalm_ k[a]$ that represents multiplication of two $C^{( k)}$ series, i.e., if ${\bf u}$ is a vector of Chebyshev expansion coefficients of $u(x)$, then $\mcalm_0[a]{\bf u}$ returns the Chebyshev expansion coefficients of $a(x)u(x)$, and  $\mcalm_ k[a]\mcals_{ k-1}\cdots\mcals_0{\bf u}$ returns the $C^{( k)}$ expansion coefficients of $a(x)u(x)$. Suppose that $a(x)$ has the Chebyshev expansion $$a(x)=\sum_{j=0}^\infty a_jT_j(x).$$ Then $\mcalm_0[a]$ can be written as  \cite{olver2013fast}:
$$\mcalm_0[a]=\frac{1}{2}\l[\begin{array}{ccccc}2a_0 & a_1 & a_2 & a_3 & \cdots\\ a_1 &2a_0 & a_1 & a_2 &\ddots \\ a_2 & a_1 & 2a_0 & a_1 & \ddots \\ a_3 & a_2& a_1& 2a_0& \ddots \\ \vdots & \ddots &\ddots &\ddots &\ddots \end{array}\r]+\frac{1}{2}\l[\begin{array}{ccccc} 0 & 0 & 0 & 0 & \cdots\\ a_1 &a_2 & a_3 & a_4 &\cdots \\ a_2 & a_3 & a_4 & a_5 & \iddots \\ a_3 & a_4& a_5& a_6& \iddots \\ \vdots & \iddots &\iddots &\iddots &\iddots \end{array}\r].$$ It follows from $$\mcalm_k[a]\mcals_{k-1}=\mcals_{k-1}\mcalm_{k-1}[a]$$ that
\beq\label{mulcon}\mcalm_k[a]\mcals_{k-1}\mcals_{k-2}\cdots\mcals_0=\mcals_{k-1}\mcals_{k-2}\cdots\mcals_0\mcalm_0[a].\eeq
The explicit formula for the entries of $\mcalm_ k[a]$ with $ k\geq 1$ is given in \cite{olver2013fast}, which uses the formula given in \cite{carlitz1961produ}.  These multiplication operators $\mcalm_ k[a]$ with $ k\geq 0$  look dense; however, if $a(x)$ is approximated by a truncation of its Chebyshev or $C^{( k)}$ series, then $\mcalm_ k[a]$ is banded. 

Combining the differentiation, conversion and multiplication operators yields \beq\label{dcode}\mcall{\bf u}=\mcals_{m-1}\mcals_{m-2}\cdots\mcals_0{\bf f},\eeq where $$\mcall:=\mcald_m+\sum_{k=1}^{m-1}\mcals_{m-1}\mcals_{m-2}\cdots\mcals_{k}\mcalm_k[a^{k}]\mcald_k+\mcals_{m-1}\mcals_{m-2}\cdots\mcals_0\mcalm_0[a^0],$$ ${\bf u}$ and ${\bf f}$ are column vectors of Chebyshev expansion coefficients of $u(x)$ and $f(x)$, respectively. Assume that the linear constraint operator $\mcalb$ in (\ref{mlc}) is given in terms of the Chebyshev coefficients of $u$, i.e., \beqs\label{dlc}\mcalb{\bf u}={\bf b}.\eeqs For example, for Dirichlet boundary conditions $$\mcalb=\l[\begin{array}{cccc}T_0(-1) & T_1(-1) & T_2(-1) & \cdots \\ T_0(1) & T_1(1) & T_2(1) & \cdots \end{array}\r]=\l[\begin{array}{ccccc} 1 & -1 & 1 & -1 &\cdots\\ 1 & 1 & 1 & 1 & \cdots \end{array}\r],$$ and for Neumann conditions $$\mcalb=\l[\begin{array}{cccc}T_0'(-1) & T_1'(-1) & T_2'(-1) & \cdots \\ T_0'(1) & T_1'(1) & T_2'(1) & \cdots \end{array}\r]=\l[\begin{array}{cccccc} 0 & 1 & -4 & \cdots & (-1)^{j+1}j^2 & \cdots\\ 0 & 1 & 4 & \cdots & j^2& \cdots \end{array}\r].$$    
%We need truncate the operators $\mcall$ and $\mcalb$ to derive a practical numerical scheme. 
Let $\mcalp_n$ be the projection operator given by $$\mcalp_n=\l[\begin{array}{cc}  {\bf I}_n & {\bf 0}  \end{array}\r].$$ We obtain the following linear system \beqs\label{mode}{\bf A}_n{\bf u}_n:=\l[\begin{array}{c} \mcalb\mcalp_n^\rmt \\ \mcalp_{n-m}\mcall\mcalp_n^\rmt \end{array}\r] \mcalp_n{\bf u} =\l[\begin{array}{c} {\bf b}\\ \mcalp_{n-m}\mcals_{m-1}\cdots\mcals_0{\bf f} \end{array}\r].\eeqs The solution $u(x)$ of (\ref{code})-(\ref{mlc}) is then approximated by the $n$-term Chebyshev series: $$u(x)\approx\sum_{j=0}^{n-1}u_jT_j(x).$$

The condition number of the matrix ${\bf A}_n$ grows like $\mcalo(n)$. Define the diagonal preconditioner $$\mcalr=\frac{1}{2^{m-1}(m-1)!}\diag\l\{{\bf I}_m, \frac{1}{m},\frac{1}{m+1},\cdots\r\}.$$ It was proved in \cite{olver2013fast} that $$ \mcalc=\l[\begin{array}{c}\mcalb\\ \mcall \end{array}\r]\mcalr-\mcali$$ is a compact operator. Therefore, the preconditioner $\mcalr$ results in a linear system with a bounded condition number.
 
\section{A Chebyshev spectral method for the integral reformulation} In this section, we propose a Chebyshev spectral method for the integral reformulation (\ref{intref2}). We refer to \cite{clenshaw1957numer,fox1961cheby,elgendi1969cheby,zebib1984cheby} for Chebyshev spectral methods for differential, integral and integro-differential equations.   

Representing $v(x)$ and $\p_x^{-k} v(x)$ by Chebyshev series \beqas v(x)=\sum_{j=0}^\infty v_jT_j(x),\qquad \p_x^{-k} v(x) =\sum_{j=0}^\infty v^{(-k)}_jT_j(x),\eeqas we have (see \cite{greengard1991spect}) $${\bf v}^{(-k)}=\mcalq^k {\bf v},$$ where 
$${\bf v}=[\begin{array}{ccc}v_0 & v_1 & \cdots \end{array}]^\rmt,\qquad {\bf v}^{(-k)}=[\begin{array}{ccc}v_0^{(-k)} & v_1^{(-k)} & \cdots \end{array}]^\rmt,$$ and, for $i,j=0,1,\ldots,\infty$,
$$\mcalq=\l[\begin{array}{ccccccccc} 1 & -\frac{1}{4} & -\frac{1}{3} &  \frac{1}{8} &  -\frac{1}{15}& \frac{1}{24} & \cdots & \frac{(-1)^{j+1}}{(j-1)(j+1)} & \cdots
\\ 1 & 0 & -\frac{1}{2} & 0 &0&0&\cdots & \cdots & \cdots 
\\ 0 & \frac{1}{4} & 0 & -\frac{1}{4} & 0 & 0&\ddots&\ddots& \ddots  
\\ \vdots & \ddots &\ddots &\ddots&\ddots&\ddots &\ddots &\ddots& \ddots
\\ \vdots & \ddots &\ddots & \frac{1}{2i} &\ddots & -\frac{1}{2i} & \ddots & \ddots & \ddots 
\\ \vdots & \ddots &\ddots &\ddots &\ddots &\ddots& \ddots & \ddots & \ddots
\end{array}\r].$$ %$$\|\mcalq\|_2\approx 1.5479.$$ 
It is easy to show that \beq\label{dqk}\mcald_k\mcalq^k=\mcals_{k-1}\mcald_{k-1}\mcalq^{k-1}=\cdots=\mcals_{k-1}\mcals_{k-2}\cdots\mcals_0.\eeq 

In the rest of this paper, we represent $\mcala$ and $\mcalx$ in terms of  the Chebyshev coefficients of their elements. For example, $$\mcalx=\l[\begin{array}{cccc} {\bf x}^0& {\bf x}^1 & \cdots & {\bf x}^{m-1} \end{array}\r].$$ We list first five ${\bf x}^j$ below \beqas {\bf x}^0&=&\l[\begin{array}{cccccc}1 & 0 & \cdots &  \end{array}\r]^\rmt,\\ {\bf x}^1&=&\l[\begin{array}{cccccc}0 & 1 & 0 & \cdots \end{array}\r]^\rmt,\\{\bf x}^2&=&\l[\begin{array}{cccccc}\frac{1}{2} & 0 & \frac{1}{2} & 0 & \cdots \end{array}\r]^\rmt,\\ {\bf x}^3&=&\l[\begin{array}{cccccc}0 & \frac{3}{4} & 0 & \frac{1}{4} & 0 & \cdots \end{array}\r]^\rmt,\\ {\bf x}^4&=&\l[\begin{array}{ccccccc} \frac{3}{8} & 0 & \frac{1}{2} & 0 & \frac{1}{8} & 0 & \cdots \end{array}\r]^\rmt.\eeqas
The integral equation (\ref{intref2}) can be rewritten as \beq\label{scop1}\wt\mcall{\bf v}={\bf f}-\mcala(\mcalb\mcalx)^{-1}{\bf b},\eeq where $$\wt\mcall:=\mcali+\sum_{k=0}^{m-1}\mcalm_0[a^k]\mcalq^{m-k}-\mcala(\mcalb\mcalx)^{-1}\mcalb\mcalq^m.$$ 

%\begin{remark}\label{wellsm} 
The equations $(\ref{dcode})$ and $(\ref{scop1})$ are related as follows.
Substituting $${\bf u}=\mcalq^m{\bf v}+\mcalx(\mcalb\mcalx)^{-1}(\bf b-\mcalb\mcalq^m{\bf v})$$ into $(\ref{dcode})$, and utilizing $(\ref{mulcon})$, $(\ref{dqk})$, and $$\mcall\mcalx=\mcals_{m-1}\mcals_{m-2}\cdots\mcals_0\mcala,$$ we obtain 
\beqs \mcals_{m-1}\mcals_{m-2}\cdots\mcals_0\wt\mcall{\bf v}=\mcals_{m-1}\mcals_{m-2}\cdots\mcals_0\l({\bf f}-\mcala(\mcalb\mcalx)^{-1}{\bf b}\r),\eeqs which can also be obtained by multiplying $(\ref{scop1})$ by $\mcals_{m-1}\mcals_{m-2}\cdots\mcals_0$ from the left. 
%\end{remark}

By truncating (\ref{scop1}), we obtain the following linear system 
\beqa\label{sc}\wt {\bf A}_n{\bf v}_n:=\mcalp_n\wt\mcall\mcalp_n^\rmt \mcalp_n{\bf v} =\mcalp_n\l({\bf f}-\mcala(\mcalb\mcalx)^{-1}{\bf b}\r).\eeqa After solving (\ref{sc}), we obtain $${\bf u}_n=\mcalp_n{\bf u}\approx \mcalp_n\mcalq^m\mcalp_n^\rmt{\bf v}_n+\mcalp_n\mcalx(\mcalb\mcalx)^{-1}\l({\bf b}-\mcalb\mcalq^m\mcalp_n^\rmt{\bf v}_n\r).$$
The solution $u(x)$ of (\ref{code})-(\ref{mlc}) is then approximated by the $n$-term Chebyshev series: $$u(x)\approx\sum_{j=0}^{n-1}u_jT_j(x).$$

The Chebyshev spectral scheme converges at the same rate as $\mcalp_n^\rmt\mcalp_n \bf v$ converges to $\bf v$. The well-conditioning of $\wt {\bf A}_n$ is obvious. Assume that the variable coefficients $a^k(x)$ can be accurately approximated by low-order polynomials. The almost banded structure of $\wt{\bf A}_n$ follows from the almost banded structure of the spectral integral operator $\mcalq$ and the banded structure of the multiplication operator $\mcalm_0[a^k]$. The number of dense rows of $\wt{\bf A}_n$ is about $\max(j,m)$, where $j$ is the number of the Chebyshev coefficients needed to resolve the variable coefficients $a^k(x)$. The adaptive QR approach \cite{olver2013fast,olver2014pract} still applies to (\ref{scop1}). %and the computational complexity is $\mcalo(j^2n)$. 
Since $\mcalm_0[a^k]$ is a Toeplitz plus an almost Hankel operator and $\mcalq$ is an almost banded operator, then the matrix-vector product for the matrix $\wt{\bf A}_n$ and an $n$-vector can be obtained in $\mcalo(n\log n)$ operators. Therefore, if $j$ is large, iterative solvers can be employed to achieve computational efficiency.

\section{Numerical experiments} 
In this section, computational results are reported for three examples (all from \cite{olver2013fast}) to test the ultraspherical spectral {\rm(US)} method, the diagonally preconditioned ultraspherical spectral {\rm(P-US)} method, and the Chebyshev spectral {\rm(CS)} method. %Our focus is on the condition numbers of the matrices ${\bf A}_n$ and $\wt{\bf A}_n$. For example with small bandwidth variable coefficient problems, we also plot the sparsity patterns. For large bandwidth variable coefficient problem, we compare the number of iterations for solving the linear systems via BiCGSTAB with TOL$=10^{-9}$.  
All computations are performed with MATLAB and the Chebfun software system \cite{driscoll2014chebf}.

\subsection*{Example 1}
Consider the linear differential equation $$u'(x)+x^3u(x)=100\sin(20,000x^2),\qquad u(-1)=0.$$ The exact solution is $$u(x)=\exp\l(-\frac{x^4}{4}\r)\int_{-1}^x 100\exp\l(\frac{t^4}{4}\r)\sin(20,000t^2)\rmd t.$$ %The computed solution approximating the exact solution to essentially machine precision is a polynomial of degree about 20,400. 

\begin{figure}[!htpb]
\centerline{\epsfig{figure=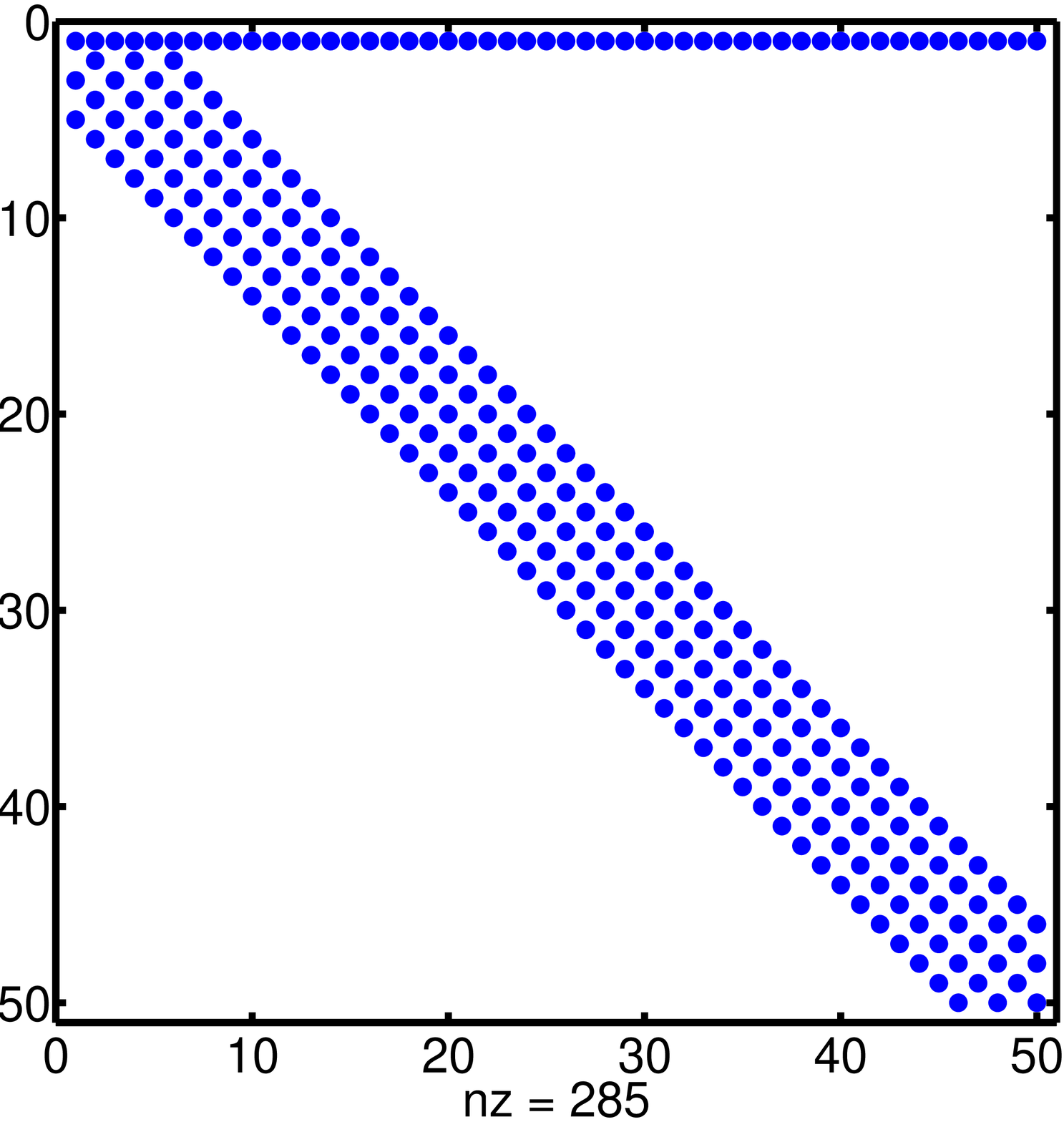,height=2.2in}%\hspace{-4mm}
\epsfig{figure=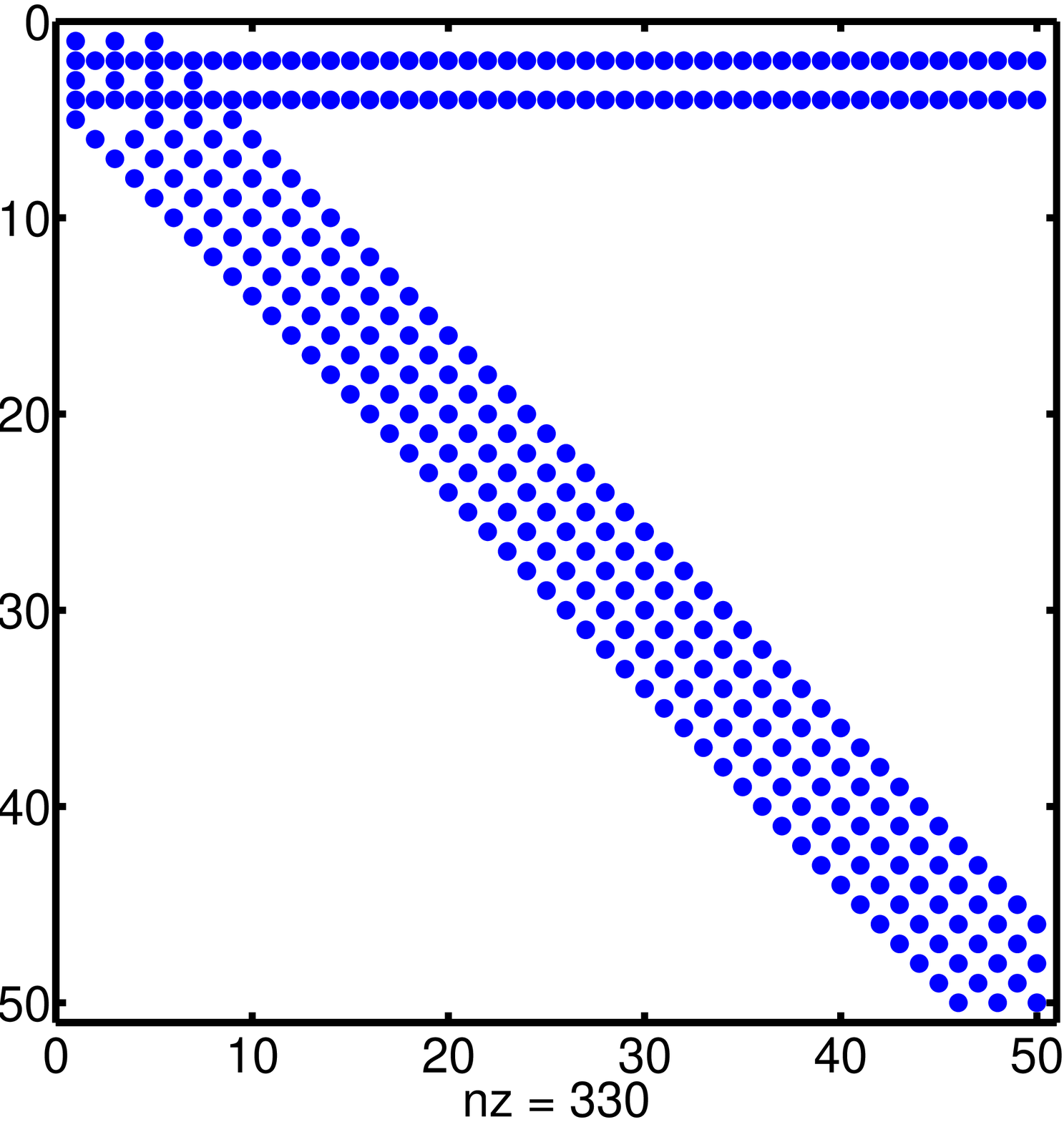,height=2.2in}}
\caption{Sparsity patterns of ${\bf A}_n$ {\rm(left)} in the ultraspherical spectral method and $\wt{\bf A}_n$ {\rm (right)} in the Chebyshev spectral method for Example $1$.}  \label{fig1}
\end{figure} 

\begin{table}[htp]
\caption{Comparison of condition numbers of the matrices for Example $1$.}
\label{e1t1}
\begin{center} \footnotesize
\begin{tabular}{c|c|c|c} \toprule
 $n$& US & P-US& CS \\ \hline
 128& 2.4045e+02 & 3.9813 &  2.5955\\
 256& 4.8312e+02 & 3.9864 &  2.5955\\
 512& 9.6846e+03 & 3.9889 &  2.5955\\
1024& 1.9391e+04 & 3.9901 &  2.5955\\ \bottomrule
\end{tabular}
\end{center}
\end{table}

In Figure \ref{fig1} we plot the sparsity patterns of the matrices ${\bf A}_n$ and $\wt{\bf A}_n$ when $n = 50$ to show the almost banded structures. We report the condition numbers in Table \ref{e1t1} and observe that the condition number of the US method behaves like $\mcalo(n)$, while those of the P-US method and the CS method remain a constant. The computed oscillatory solution and its first-order derivative by the CS method are plotted in Figure \ref{fig2}. The $L^2$ norm errors for the computed first-order derivative approximating the exact first-order derivative to machine precision and the computed solution by the CS method are $$\l(\int_{-1}^1(u'(x)-\wt u'(x))^2\rmd x\r)^{\frac{1}{2}}=0,$$ and $$\l(\int_{-1}^1(u(x)-\wt u(x))^2\rmd x\r)^{\frac{1}{2}}=1.2602\times 10^{-14},$$ respectively.

\begin{figure}[!htpb]
\centerline{\epsfig{figure=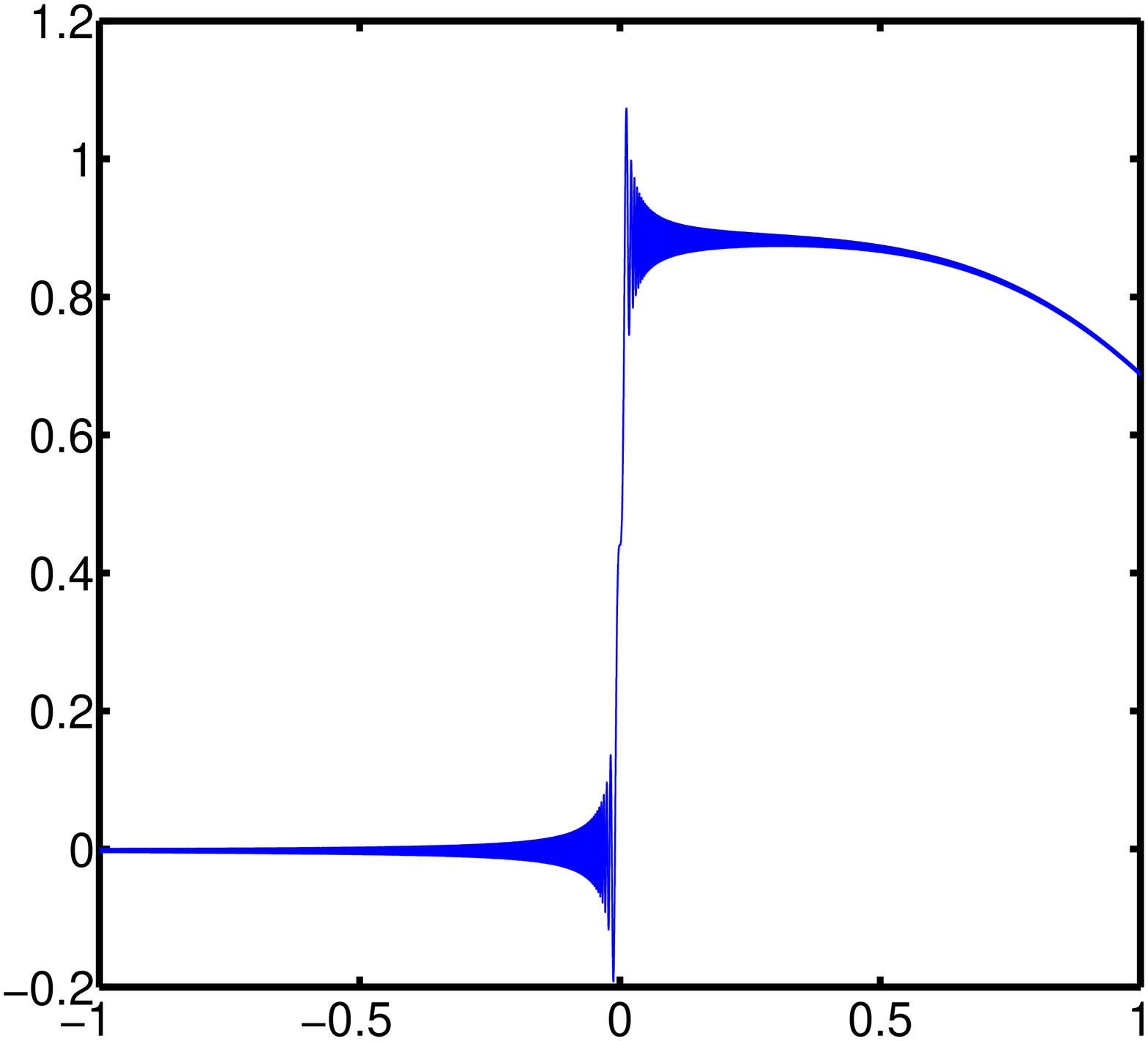,height=2.2in}%\hspace{-4mm}
\epsfig{figure=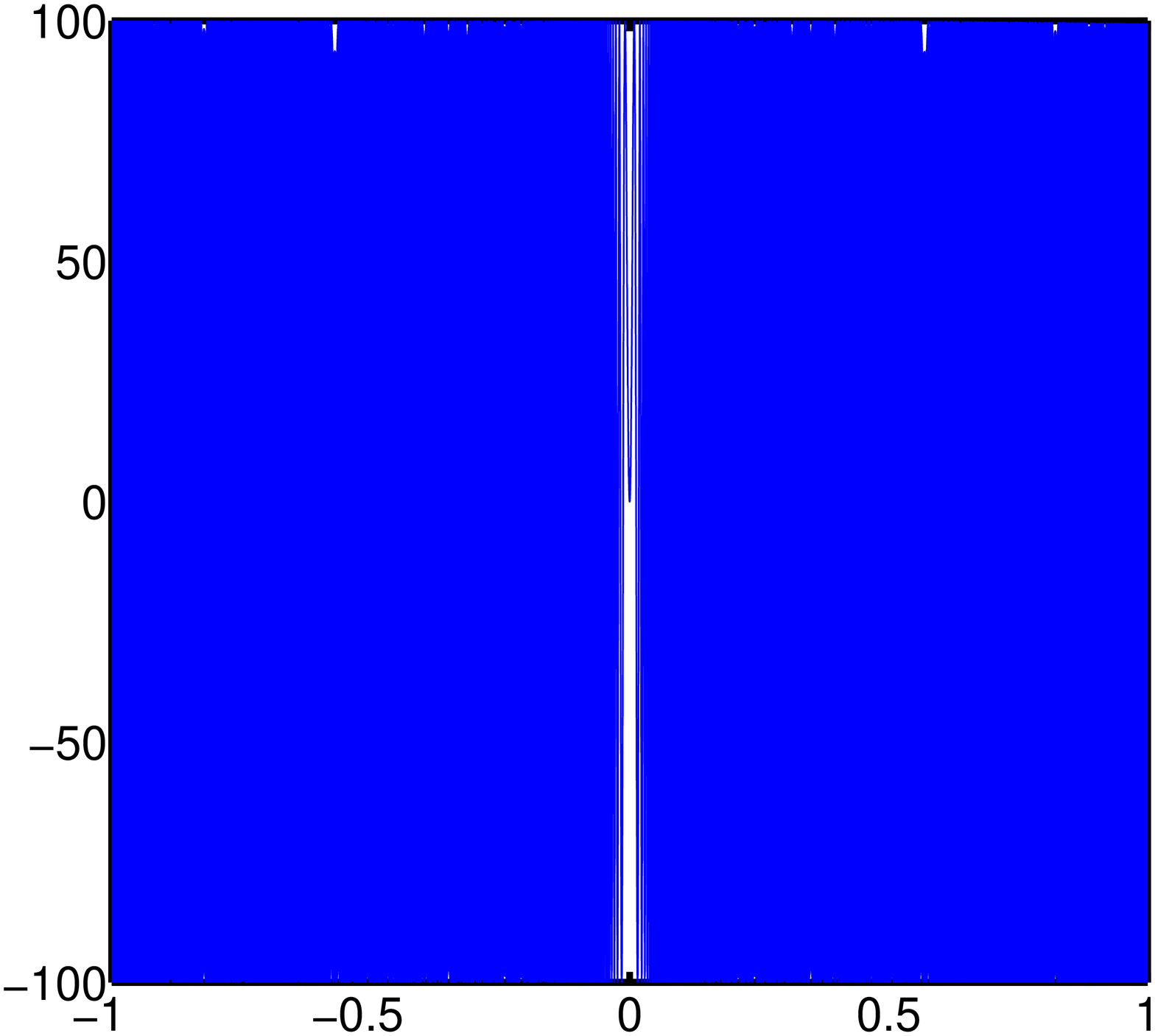,height=2.2in}}
\caption{The highly oscillatory solution {\rm(left)} and its first-order derivative {\rm (right)} computed by the Chebyshev spectral method for Example $1$.}  \label{fig2}
\end{figure}

\subsection*{Example 2} Consider the linear differential equation $$u'(x)+\frac{1}{ax^2+1}u(x)=0,\qquad u(-1)=1.$$ The exact solution is $$u(x)=\exp\l(-\frac{\arctan(\sqrt{a}x)+\arctan(\sqrt{a})}{ \sqrt{a} }\r).$$ We take $a=50,000$. The variable coefficient function can be approximated to roughly machine precision by a polynomial of degree $7315$. The multiplication operator $\mcalm_0[a^0]$ has a very large bandwidth, resulting in essentially dense linear systems.

We compare condition numbers (in Table \ref{e2t1}), number of iterations (using Bi-CGSTAB in Matlab with TOL$=10^{-14}$, in Figure \ref{e2fig1} (left)), and $L_2$ norm errors (in Figure \ref{e2fig1} (right)) of US, P-US and CS. Observe from Table \ref{e2t1} that the condition number of the US method behaves like $\mcalo(n)$, while those of the P-US method and the CS method remain a constant even for $n$ up to $8192$. As a result, the P-US method and the CS method only require several iterations to converge (see Figure \ref{e2fig1} (left)), while the usual US scheme requires much more iterations with a degradation of accuracy as depicted in Figure \ref{e2fig1} (right). 

\begin{table}[htp]
\caption{Comparison of condition numbers of the matrices for Example $2$.}
\label{e2t1}
\begin{center} \footnotesize
\begin{tabular}{c|c|c|c} \toprule
 $n$& US & P-US& CS \\ \hline
1024& 1.7953e+03 & 3.3256 &  1.0912\\
2048& 3.5925e+03 & 3.3261 &  1.0912\\
4096& 7.1870e+03 & 3.3264 &  1.0912\\
8192& 1.4376e+04 & 3.3265 &  1.0912\\ \bottomrule
\end{tabular}
\end{center}
\end{table}

\begin{figure}[!htpb]
\centerline{\epsfig{figure=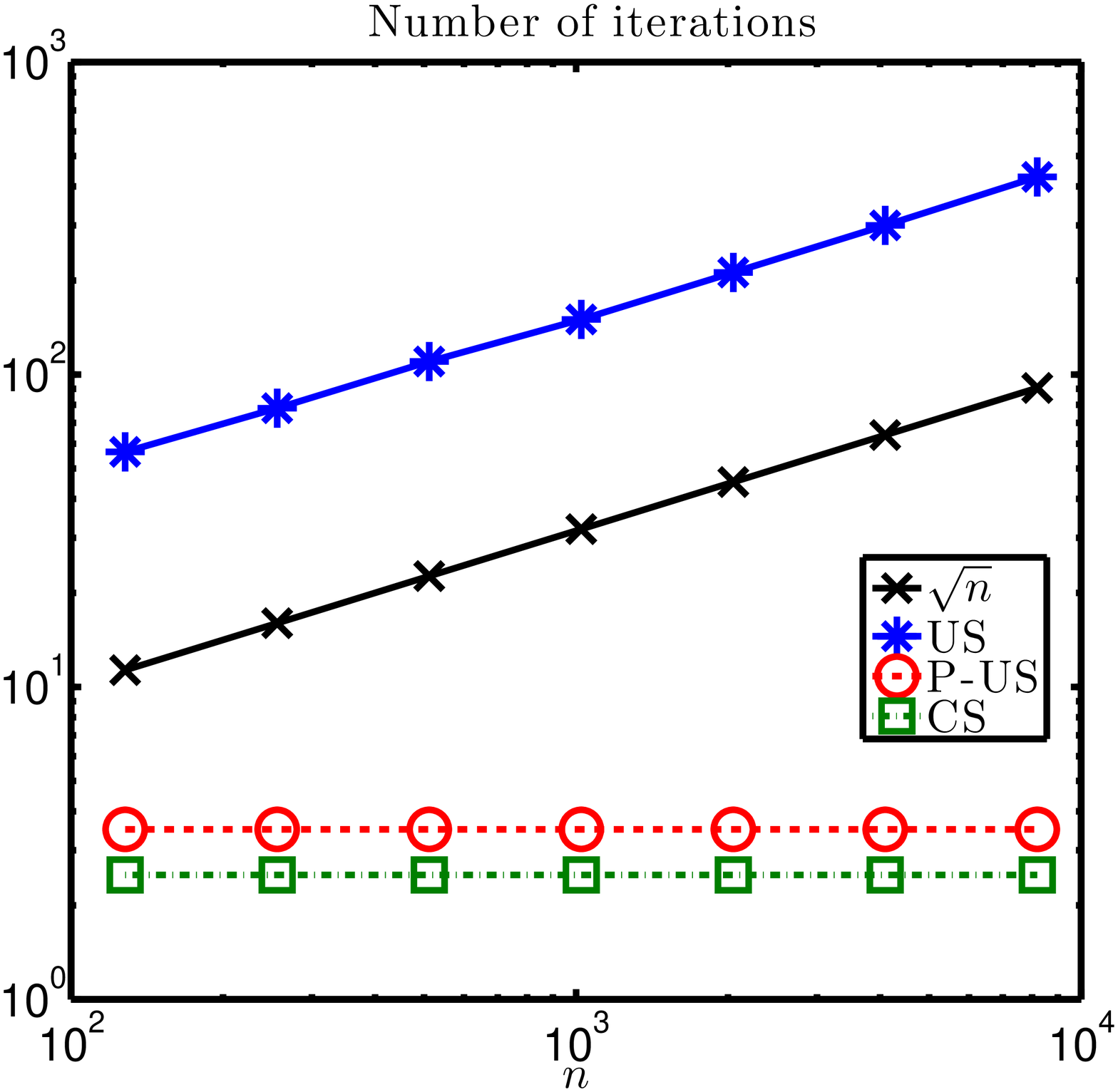,height=2.1in}%\hspace{-4mm}
\epsfig{figure=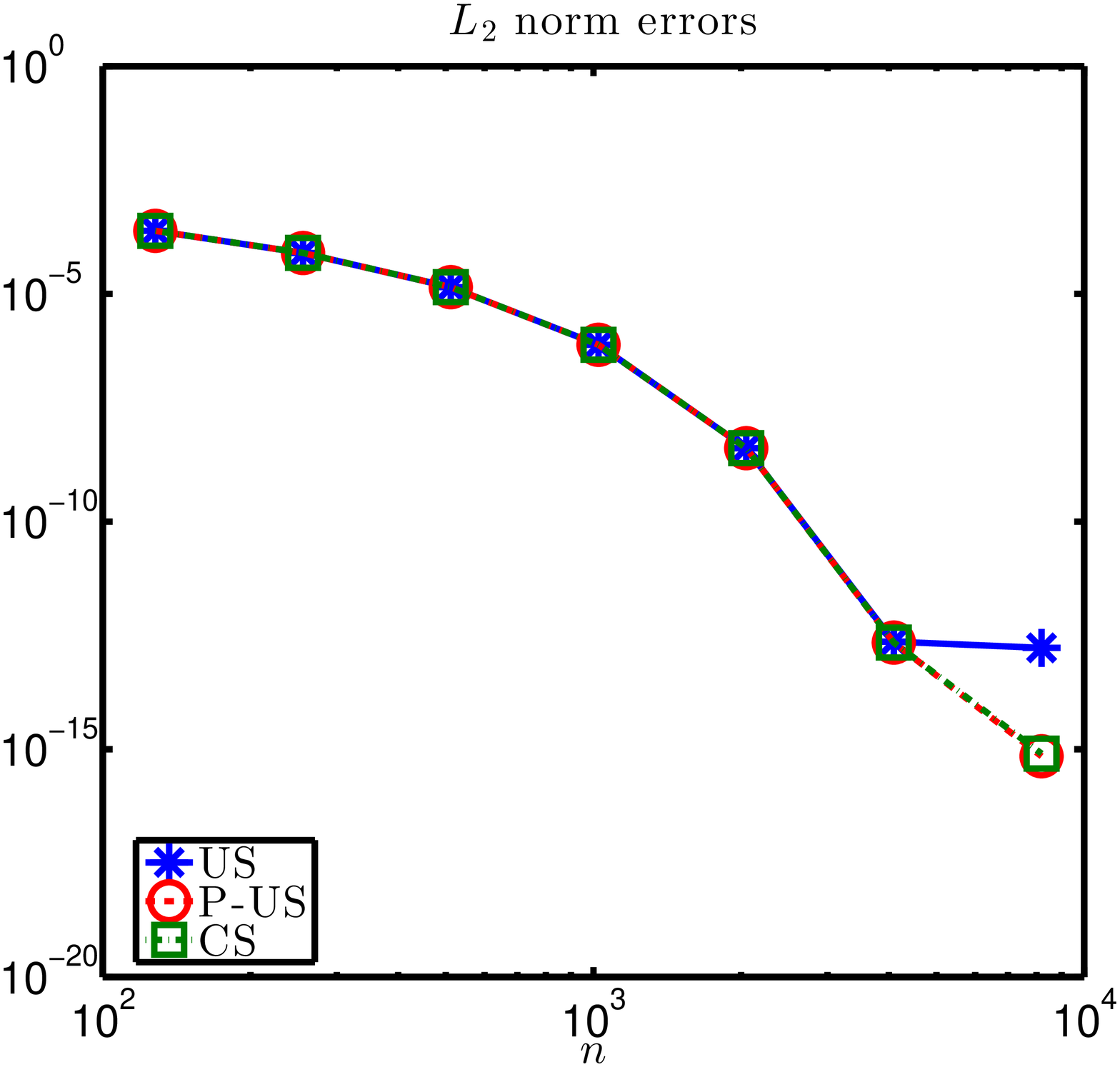,height=2.1in}}
\caption{Number of iterations {\rm(left)} and $L_2$ norm errors {\rm(right)} by the ultraspherical spectral method and the Chebyshev spectral method for Example $2$.}  \label{e2fig1}
\end{figure}

\subsection*{Example 3} Consider the high order differential equation $$u^{(10)}(x)+\cosh(x)u^{(8)}(x)+x^2u^{(6)}(x)+x^4u^{(4)}+\cos(x)u^{(2)}(x)+x^2u(x)=0,$$ with boundary conditions $$u(\pm1)=0,\qquad u'(\pm1)=1,\qquad u^{(k)}(\pm1)=0,\qquad 2\leq k\leq 4.$$

In Figure \ref{fig3} we plot the sparsity pattern of the matrix $\wt{\bf A}_n$ when $n = 100$ and the computed solution by the Chebyshev spectral method. We also note that the condition number of $\wt{\bf A}_n$ remains a constant (about $1.7444$) for different values of $n$. The computed solution is odd to about machine precision, $$\l(\int_{-1}^1(u(x)-\wt u(x))^2\rmd x\r)^{\frac{1}{2}}=5.2171\times 10^{-14}.$$

\begin{figure}[!htpb]
\centerline{\epsfig{figure=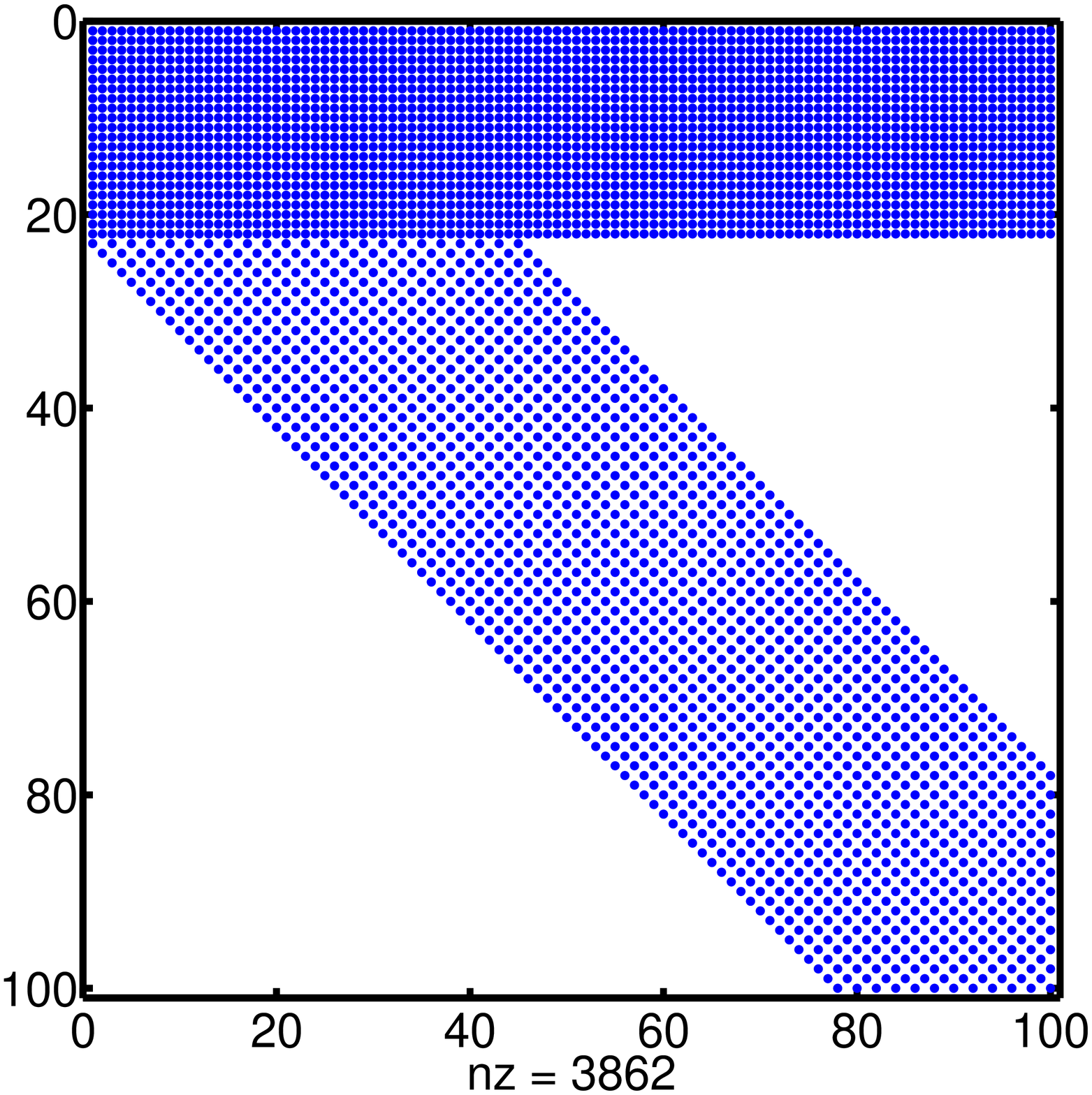,height=2.2in}%\hspace{-4mm}
\epsfig{figure=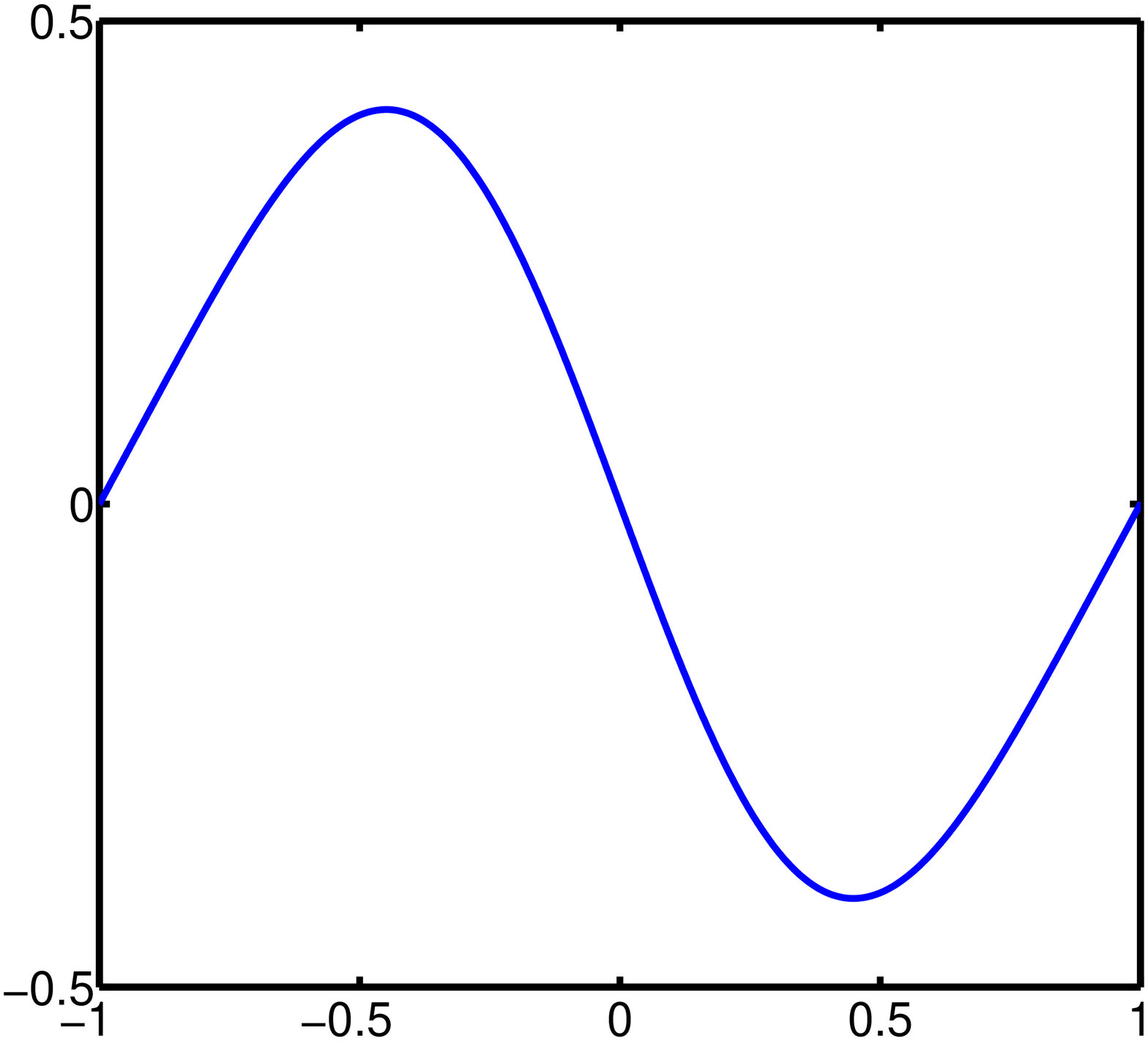,height=2.2in}}
\caption{The sparsity pattern of $\wt{\bf A}_n$ {\rm (left)} and the computed solution {\rm(right)} by the Chebyshev spectral method for Example $3$.}  \label{fig3}
\end{figure} 

\section{Concluding remarks} We have revisited the well-conditioned spectral methods \cite{driscoll2010autom,wang2014well} and the ultraspherical spectral method \cite{olver2013fast} from the viewpoint of the integral reformulation (\ref{intref2}). We also proposed a Chebyshev spectral method for the integral reformulation (\ref{intref2}), which preserves the almost banded structure, avoids the conversion operators $\mcals_k$ and only needs the multiplication operators $\mcalm_0[a^k]$. Therefore, the Chebyshev spectral method is very easy to implement.  The well-conditioning of these methods come from that of the integral operator. The integral reformulation approach can also be used to interpret the well-conditioning of the fractional spectral collocation methods \cite{jiao2015well,du2015prefsc}.
However, we have to mention that although it is independent of the discretization parameter, the condition number of the coefficient matrix may be very large. For example, see the singular perturbation problem with a tiny parameter \cite{olver2013fast,du2015prersc}. 

Newton iteration techniques and the tensor-product techniques \cite{julien2009effic,townsend2013exten,townsend2015autom,du2015two} can be used in the extensions of this work to nonlinear problems and high-dimensional problems,  respectively. Recently, Shen, Wang and Xia \cite{shen2015fast} proposed a fast structured direct spectral method for differential equations with variable coefficients by employing the low-rank property of the coefficient matrix. The computational complexity of their method is nearly linear. Numerical experiments (including variable coefficients with steep gradients) show that the low-rank property still holds for the coefficient matrix in the Chebyshev spectral method proposed in this work. Theoretical explanation for this point is being investigated and will be reported elsewhere.

\section*{Acknowledgments} The author would like to thank Prof. Jan S. Hesthaven for  valuable comments on the manuscript \cite{du2015prersc}. The author also thanks Dr. Can Huang for the discussion about compact operators and fractional spectral collocation methods. 

\end{document}